\documentclass[12pt]{article}

\usepackage[english]{babel}
\usepackage[utf8]{inputenc}
\usepackage{amsmath}
\usepackage{graphicx}
\usepackage{amssymb}
\usepackage{amsthm}
\usepackage{tikz-cd}
\usepackage{mathrsfs}
\usepackage[colorinlistoftodos]{todonotes}
\usepackage{enumitem}
\usepackage{yfonts}
\usepackage{xcolor}
\usepackage{mathtools}
\usepackage{hyperref}

\title{On the minimal size of a generating set for Lattices in Lie Groups}
\author{Tsachik Gelander, Raz Slutsky}

\date{}
\newtheorem{thm}{Theorem}[section]
\newtheorem{lem}[thm]{Lemma}
\newtheorem{prop}[thm]{Proposition}

\newtheorem{defn}[thm]{Definition}
\newtheorem{eg}[thm]{Example}

\newtheorem{cor}[thm]{Corollary}
\newtheorem{claim}[thm]{Claim}
\newtheorem{rmk}[thm]{Remark}

\newcommand{\C}{\mathbb{C}}

\newcommand{\Sl}{SL_2(\mathbb{R})}
\newcommand{\bigslant}[2]{{\raisebox{.2em}{$#1$}\left/\raisebox{-.2em}{$#2$}\right.}}%

\begin{document}
\maketitle

\begin{abstract}
We prove that the rank (that is, the minimal size of a generating set) of lattices in a general connected Lie group is bounded by the co-volume of the projection of the lattice to the semi-simple part of the group. This was proved by Gelander for semi-simple Lie groups and by Mostow for solvable Lie groups. Here we consider the general case, relying on the semi-simple case. In particular, we extend Mostow's theorem from solvable to amenable groups.
\end{abstract}

\section{Introduction}
We denote by $d(\Gamma)$ the minimal number of generators of the group $\Gamma$.
A result by Gelander \cite{gelander11} states the following.

\begin{thm}[Gelander, 2011]
\label{gelander2011}
Let $G$ be a connected semisimple Lie group without compact factors. Then there exists a constant\footnote{The constant $C$ depends of course on the normalization of the Haar measure of $G$.} $C=C(G)$ such that \[ d(\Gamma) \leq C \operatorname{vol}(G/\Gamma) \] for every irreducible lattice $\Gamma \leq G$.
\end{thm} 

On the other hand, Mostow \cite{mostow54} showed that the rank of lattices in solvable Lie groups is bounded uniformly by the dimension of the group.\\

Our aim is to establish a similar result for a general Lie group without the semi-simplicity assumption. However, the following example shows we can not wish for a straightforward analogue.

\begin{eg}
Let $G = \mathbb{R} \times \Sl$. Now define the lattice $\Gamma_0 = \alpha \mathbb{Z} \times \Gamma$ for some $\alpha >0$ and $\Gamma \leq \Sl$ a lattice in $\Sl$. 
\end{eg}

While we can take $\Gamma$ with arbitrarily high rank, $\alpha$ can be taken to be small, to make the volume of $\mathbb{R}/\alpha \mathbb{Z}$ as small as we want.
Since \[ \text{vol}(G/\Gamma_0) =  \text{vol}(\mathbb{R} / \alpha \mathbb{Z}) \, \text{vol}(\Sl/\Gamma) \]
we get a lattice with a rank that is not controlled by its co-volume. However, when increasing the rank of $\Gamma$,  $\text{vol}(SL_2(\mathbb{R})/\Gamma)$ must grow. 

In light of this counter-example we define the semisimple co-volume of a lattice. 
To do this, note that every Lie group admits a Levi decomposition of the form $G = S \ltimes R$ where $R$ is the solvable radical and $S$ is the semisimple part.
This is in fact an almost direct product, as the map $S \times R \rightarrow G$ may have a discrete central kernel. 
A similar decomposition to an almost direct product can be done with the amenable radical of $G$, that is, the maximal amenable normal connected subgroup of $G$. 
In this case, we have that $G = S_{nc} \ltimes A$ where $A$ is amenable and $S_{nc}$ is semi-simple with no compact factors. 
Note that $A = S_c \ltimes R$ where $S_c$ is the maximal compact factor of the semi-simple part. 
This is because solvable and compact groups, as well as products of such, are amenable, and
non-compact semi-simple groups are not.
If $\Gamma$ is a lattice in $G$, then, by Auslander's theorem, $\Gamma \cap A$ is a lattice in $A$ 
and the projection of $\Gamma$ to the semi-simple non-compact part of $G$ is a lattice as well (see Lemma \ref{amenable_intersection}) . 
We are now ready to define the semi-simple co-volume of a lattice.

\begin{defn}
Let $G$ be a Lie group, and let $G = S \ltimes A$ be its decomposition, with $A$ being the amenable radical. Then the semisimple co-volume of $\Gamma$ is defined to be $\operatorname{covol}_s(\Gamma) := \operatorname{vol}(S/ \pi(\Gamma))$ where $\pi : G \rightarrow G/A$ is the projection.
\end{defn}

Having this definition under our belt, we state the main result.

\begin{thm}
\label{main_theorem}
  Let $G$ be a connected Lie group, then there exist constants $a,b$, depending on $G$, such that for every  lattice $ \Gamma \leq G$  
\[ 
 d(\Gamma) \leq a \operatorname{covol}_s(\Gamma) + b.
\]
\end{thm}

As in Theorem \ref{gelander2011}, the constant $a$ depends on the normalization of the Haar measure of $G/A$.

This theorem is a counter-part of the main theorem in $\cite{gelander11}$, where $G$ is assumed to be semi-simple. 

\begin{rmk}
A similar result can be stated for non-connected Lie groups, under the additional assumption that $G$ is compactly generated (equivalently, that $G/G^0$ is f.g.). If we write $\Gamma_0 := \Gamma \cap G^0$ we have that $\operatorname{vol}(G/\Gamma) = \operatorname{vol}(G^0/\Gamma_0) [G: \Gamma G^0]$ and one can define $\operatorname{covol}_s(\Gamma) := \operatorname{covol}_s(\Gamma_0)[G:\Gamma G^0]$. The analogous result then follows from the connected case.
\end{rmk}

Note that Theorem \ref{main_theorem} implies the well known but non-trivial:
\begin{cor}
Lattices in connected Lie groups are finitely generated.
\end{cor}
Originally this result was proved separately and by different methods for different types of Lie groups. The argument from \cite{gelander11} and the completions given here form a unified proof for this result.


\section{The Proof}
\subsection{Proof outline}
Na\"{i}vely, one may wish to use the Levi decomposition. 
This will enable to combine the known results for lattices in solvable and in semi-simple Lie groups. 
However, $\Gamma \cap R$ is unfortunately not necessarily a lattice in $R$, the solvable radical of $G$, in contrast with the assertion of \cite[~8.28]{raghunathan72}. 
For example, take $G = \mathbb{R} \times SO(3)$, and define $\Gamma := \langle (1,g) \rangle$ for some element $g \in SO(3)$ of infinite order. 
This is a lattice in $G$ since $[0,1) \times SO(3)$ is a fundamental domain with finite volume. 
However, $\Gamma \cap R = \Gamma \cap \mathbb{R} = \{ 0\}$ which is clearly not a lattice, and the projection of $\Gamma$ to $SO(3)$ is nondiscrete. 

To this end we decompose $G$ as $G = S_{nc} \ltimes A$ where $S_{nc}$ is the product of the noncompact simple factors of the semisimple part of the Levi decomposition, and $A$ is the amenable radical.
While the mentioned corollary of Raghunathan is not true in general as stated, it can be corrected by replacing the solvable radical with the amenable one. 
%

\begin{lem}
\label{amenable_intersection}
Let $G$ be a connected Lie group and let $A$ be its amenable radical. If $\Gamma$ is a lattice in $G$, then $\Gamma \cap A$ is a lattice in $A$ and $\pi_{G/A}(\Gamma)$ is a lattice in $G/A$.
\end{lem}

This statement is well known but for reference we sketch a proof\footnote{More elementarily, one can also use Borel's density theorem \cite[Lemma 3.4(d)]{mostow71} to show that $\overline{\Gamma R}^{\mathrm{o}} \subset RS_c$, 
 and thus obtain $\Gamma \overline{\Gamma R}^{\mathrm{o}}S_c = \overline{\Gamma R}S_c = \Gamma RS_c$,
hence $\Gamma RS_c$ is closed.
Looking at $\Gamma \subset \Gamma RS_c \subset G$ , since $G / \Gamma$ has finite volume, $\Gamma RS_c / \Gamma = RS_c / (\Gamma \cap RS_c) = A / (\Gamma \cap A)$ must also have finite volume.}.

\begin{proof}

In view of \cite[Theorems 4.3, 4.7]{onischik00}  it is suffices to show that $\pi_{G/A}(\Gamma)$ is discrete in $G/A$. 
Since $A$ is amenable and $\Gamma$ is discrete, it follows from \cite[Theorem 9.5]{toti} that the identity connected component of $\overline{\pi_{G/A}(\Gamma)}$ is amenable. 
Since $G/A$ is semisimple with no compact factors it follows from the Borel density theorem that $\overline{\pi_{G/A}(\Gamma)}^\circ$ is normal in $G/A$ and therefore trivial. 
Thus $\pi_{G/A}(\Gamma)$ is discrete.
\end{proof}

%
%
%
As shown in \cite{bcmg16}, lattices in amenable Lie groups are co-compact, hence finitely generated. This leads us to the following route of reasoning. \\

We're left to prove our result separately in the amenable radical and the semi-simple part. In particular, we need to bound the rank of lattices in amenable Lie groups uniformly, i.e. to extend Mostow's theorem from solvable to amenable groups. 
Concerning the semisimple case, in \cite{gelander11} it is assumed that the lattice is irreducible, thus we have also have to reduce it from the general case to the irreducible one.

\subsection{Lattices in amenable groups}

Consider the case $G = OR$, where $O$ is a semi-simple compact Lie group, and $R$ is a connected, normal, solvable Lie group. Let $\Gamma \leq G$ be a lattice, and
define $L := \overline{\pi_{G/R}(\Gamma)}^{\mathrm{o}}$. By \cite[8.24]{raghunathan72},  $L$ is solvable. 
However, $L$ is a connected solvable subgroup of a compact group, thus abelian.
In fact, this is true even if $O$ is not compact, since by $\cite[\text{Lemma 3.4 (d)}]{mostow71}$ $L$ lies inside the compact factor of the semi-simple part. \\

Now, define $\tilde{R} := \pi_{G/R}^{-1}(L)$. It is solvable, since $\tilde{R}/R$ is abelian. We will soon show that $\Gamma \cap \tilde{R}$ is a lattice in $\tilde{R}$. Being solvable, we know that $\Gamma \cap \tilde{R}$ is generated by at most $dim (\tilde{R})$ elements \cite[3.8]{raghunathan72}. However, $\tilde{R}$ is not normal in $G$ and so we can not quotient by it. To that end we wish to look at $\tilde{R}$ inside its normalizer $N_G(\tilde{R})$.  
\\

First, let us prove the aforementioned claim
\begin{claim}
\label{intersection_is_lattice}
Let $G = O \cdot R$, and let $\Gamma$ be a lattice. Define $\tilde{R}$ as before. Then $\Gamma$ is a lattice in $H := N_G(\tilde{R})$ and $\Gamma \cap \tilde{R}$ is a lattice in $\tilde{R}$
\end{claim}
\begin{proof}
First, since $\tilde{R}$ is closed, so is $H$. Clearly, $\Gamma$ normalizes $\tilde{R}$ so $\Gamma \subset H$. Using Mostow's Lemma, \cite{mostow62}, $H/\Gamma$ has finite volume. Turning to $\tilde{R}$, observe that $\Gamma \tilde{R} = \Gamma \overline{\Gamma R}^{\mathrm{o}} = \overline{\Gamma R}$ which is of course closed. Taking $\Gamma \subset \Gamma \tilde{R}$ and applying Mostow's Lemma again, it follows that $\Gamma \tilde{R} / \Gamma \cong \tilde{R} / (\Gamma \cap \tilde{R})$ has an invariant finite measure, which means that $\Gamma \cap \tilde{R}$ is a lattice in $\tilde{R}$.
\end{proof}

Next we wish to work with $\Gamma \cap H^{\mathrm{o}}$.
To do so, while keeping track on the rank of $\Gamma$, we prove that $\lvert H/H^{\mathrm{o}} \rvert$ is bounded uniformly for all lattices in $G$.
Recall that a compact Lie group admits a structure of a real algebraic group.

\begin{prop}
\label{bounding_normalizer_components}
Let $G$ be a reductive linear algebraic group, then there exists a constant $c = c(G)$ such that $\lvert N_G(F)/N_G(F)^{\mathrm{o}} \rvert \leq c$ for all algebraic subgroups $F \leq G$.
\end{prop}

For the proof of \ref{bounding_normalizer_components} we use the following result from Algebraic Geometry. 

\begin{thm}[Generalized Bezout's Theorem]
Let $X_1,\ldots,X_n$ be pure dimensional varieties over $\C$ (that is, every irreducible component has the same dimension), and let $Z_1,\ldots,Z_m$ be the irreducible components of $X_1 \cap \ldots \cap X_n$. Then, \[ \sum_{i=1}^{m} \text{deg} Z_i \leq \prod_{j=1}^{n} \text{deg} X_j \]

for reference see \cite[p. 519]{sch00}.
\end{thm}


\begin{proof}[Proof of Proposition \ref{bounding_normalizer_components}]


Let $\mathfrak{f}$ be the Lie algebra of $F$, and define 
$$
 H := \{ g \in G; Ad(g) \mathfrak{f} = \mathfrak{f} \}.
$$
This is exactly the normalizer of $\mathfrak{f}$. But now one can view $H$ as an algebraic variety embedded in a projective space. 
$H$ is of bounded degree, since $\mathfrak{f}$ is a vector space, thus the solutions to $Ad(g) \mathfrak{f} = \mathfrak{f}$ are defined by at most the number of polynomials defining $G$ and some bounded number of linear equations. 
As $H$ is an algebraic group, its irreducible components are exactly its connected components.
 Furthermore, each irreducible component has the same dimension, thus $H$ is pure dimensional, and viewing it as a variety over $\mathbb{C}$, the generalized Bezout's theorem can be invoked, taking $X_1 = H$. 
Since the degree of $H$ is bounded regardless of $F$, it follows that the number of irreducible components, hence the number of connected components, is uniformly bounded, as $\text{deg}(Z_i) \geq 1$. Going back to our original real algebraic group $G$, we look at the real points of $H$, and as the number of connected components of the real points of a connected reductive algebraic group is bounded by $2^{\text{Rank(G)}}$, see for example \cite[14.5]{bt65}, the result follows.
\end{proof}

\begin{rmk}
Note that in the case under consideration, where $G$ is compact, one can omit the term $2^{\text{Rank(G)}}$ in the argument above since $B$ is a compact group hence its connected components are algebraic.
\end{rmk}

 This leads to the following corollary.
\begin{cor}
\label{components_bound_cor}
Let $H = N_G(\tilde{R})$ be defined as before, then $\lvert H/H^{\mathrm{o}} \rvert \leq c(G)$.
\end{cor}

\begin{proof}
Let $\pi : G \rightarrow G/R $ be the projection, and define $N = N_{G/R}(L)$. Then $H = \pi^{-1} (N)$ and consequentially, as $R$ is connected,  $\lvert N/N^{\mathrm{o}} \rvert = \lvert H/H^{\mathrm{o}} \rvert $. Now, $N \leq G/R$ being a compact Lie group can be viewed, due to classical results of Peter--Weyl and Chevalley \cite{chev46, hall_15}, as a linear algebraic group, and so the corollary follows.
\end{proof}

Now, look at $\Gamma_{H^{\mathrm{o}}} := \Gamma \cap H^{\mathrm{o}}$. Let $\pi_{H^\circ/\tilde{R}} : H^{\mathrm{o}} \rightarrow H^{\mathrm{o}}/ \tilde{R}$ be the natural projection, then $\pi_{H^\circ/\tilde{R}}(\Gamma_{H^{\mathrm{o}}})$ is finite, since by Claim \ref{intersection_is_lattice} the intersection is a lattice, and thus also the projection, and we have that $R \leq \tilde{R}\le H^{\mathrm{o}} \leq G$, while $G/R$ is compact, hence this projection is discrete in compact, hence finite. We are left to show that the image has bounded rank which is independent of $\Gamma$. To this end we prove the following lemma.

\begin{prop}
\label{rank_of_finite_subgroups}
For $n \in \mathbb{N}$ there exists a constant $c = c(n)$ such that for every finite subgroup $H \leq GL_n(\mathbb{C})$ it holds that $d(H) \leq c$.
\end{prop}

\begin{proof}
By the Jordan--Schur theorem, there is a constant $k = k(n)$ such that given a finite subgroup $H$ of $GL_n(\mathbb{C})$ there exists an abelian subgroup, $A \trianglelefteq H$ such that $\lvert H:A \rvert \leq k$. Now,  every $T \in A$ satisfies $T^{\lvert A \rvert} = \text{Id}$ and so its minimal polynomial divides $x^{\lvert A \rvert}-1$ which has distinct roots, hence $T$ is diagonalizable. By an elementary result in linear algebra, a finite set of commuting diagonalizable matrices is simultaneously diagonalizable, and so $A$ is isomorphic to a finite diagonal subgroup. Such a group is a subgroup of the direct sum of $n$ cyclic groups, hence generated by at most $n$ elements. Consequently, we have that $d(H) = n + \lvert H:A \rvert \leq n + k$. In \cite{collins07} a concrete bound on $k$ is given by $k \leq (n+1)!$.
\end{proof}

\begin{cor}
\label{finite_subgroups_of_compact}
Let $k \in \mathbb{N}$, then there exists a constant $m = m(k)$ such that for every connected $k$-dimensional compact Lie group $G$ and every finite subgroup $H \leq G$ it holds that $d(H)\leq m$.
\end{cor}

\begin{proof} We show that there are only finitely many isomorphism classes of connected compact Lie groups of a given dimension. Thus, there exists $n = n(k)$ such that every $k$ dimensional connected compact Lie group embeds in $GL_n(\mathbb{C})$. The result then follows from the previous proposition. 

By \cite[13, Thm 1.3]{hochschild65}, every compact connected Lie group of dimension $k$ is  of the form $G = (TS)/D$ where $T$ is a torus, $S$ is compact semi-simple and $D$ is a finite central subgroup, such that $D \cap T, D \cap S$ are trivial. This implies that the projection of $D$ to $S$ is injective, and the image is of course central. The center of compact semi-simple groups is finite, thus for each $T,S$ there are only finitely many such quotients. Due to the classification of semi-simple groups there is also only a finite number of such $S$ of a given dimension. 
\end{proof}

Specializing to our case, we get the following corollary.

\begin{cor}
\label{finite_projection_bound}
$d(\pi_{\tilde{R}}(\Gamma_{H^{\mathrm{o}}}))$ is uniformly bounded regardless of the choice of $\Gamma$.
\end{cor}

\begin{proof}
$H^{\mathrm{o}} / \tilde{R}$ is a compact Lie group of bounded dimension, hence we can apply Corollary \ref{finite_subgroups_of_compact}. 
\end{proof}

We conclude this part with the following theorem.

\begin{thm}
\label{rank_in_amenable_radical_is_bounded}
Let $A$ be a connected amenable Lie group.
Then there exists a constant $b =b(A)$ such that $d(\Gamma) \leq b$ for all lattices $\Gamma \leq A$.
\end{thm}

\begin{proof}
Let $A = O \ltimes R$ be the Levi decomposition of $G$, where $R$ is connected, normal and solvable, and $O$ is compact, connected and semi-simple.
Take $H = N_{A}(\tilde{R})$ as before. In \ref{components_bound_cor} we showed that the number of connected components is uniformly bounded, so we can assume $H$ is connected. 
The rank of $\Gamma_{H^{\mathrm{o}}} = \Gamma$ is then bounded uniformly due to the following exact sequence 

\[
1 \rightarrow \Gamma \cap \tilde{R} \rightarrow \Gamma \rightarrow \bigslant{\Gamma}{\Gamma \cap \tilde{R}} \rightarrow 1,
\]
which implies that
$$
 d(\Gamma) \leq d(\Gamma \cap \tilde{R}) + d(\bigslant{\Gamma}{\Gamma \cap \tilde{R}}).
$$ 
Now the first summand is uniformly bounded since $\Gamma \cap \tilde{R}$ is a lattice in a solvable Lie group, hence generated by at most dim$(\tilde{R})$ elements \cite[3.7]{raghunathan72}. The second summand is uniformly bounded by Corollary \ref{finite_projection_bound}.
\end{proof}


\subsection{Lattices in semi-simple Lie groups}

Since Theorem \ref{gelander2011} was stated and proved in \cite{gelander11} for irreducible lattices, we restate and deduce it here without this assumption. 

\begin{thm}
\label{reducible_generalization}
Let $G$ be a connected semi-simple Lie group with no compact factors, then there exists a constant $C= C(G)$ such that \[ d(\Gamma) \leq C \operatorname{vol}(G/\Gamma) \]
for every lattice $\Gamma \leq G$.
\end{thm}




\begin{lem}
Let $\Gamma \leq G_1 \times G_2$ be a lattice and assume $\Gamma_1 := \Gamma \cap G_1$ is a lattice in $G_1$. Now let $\pi : G \rightarrow G/G_1$ be the projection and denote $\Gamma_2 := \pi(\Gamma)$. Then $\operatorname{covol}(\Gamma) = \operatorname{vol}(G_1 / \Gamma_1) \operatorname{vol}(G_2/ \Gamma_2)$
\end{lem}

\begin{proof}
We claim that some fundamental domain $D$ for $G/\Gamma$ can be viewed as $D = D_1 \times D_2$ where $D_i$ is a fundamental domain of $\Gamma_i \leq G_i$. This is evident as we can shift every element $(g_1,g_2)$ to $D$ uniquely by first multiplying by an element of $G_2$ and then by an element of $G_1$ which would not change the second coordinate. Since $m_G = m_{G_1} \times m_{G_2}$ is a Haar measure on $G$ the result follows.
\end{proof}

\begin{proof}[Proof of Theorem \ref{reducible_generalization}]
First note that by replacing $G$ with its simply connected cover we may suppose that $G$ is simply connected.
Suppose that $\Gamma\le G$ is not irreducible. Then there is a non-trivial decomposition $G=G_1\times G_2$ such that $\Gamma\cap G_i$ is a lattice in $G_i$ for $i=1,2$.

Denote $\Gamma_1 := \Gamma \cap G_1$ and $\Gamma_2 := \pi_{G_2}(\Gamma)$ where $\pi_{G_2}: G \rightarrow G_2$ is the projection. 
Since $G_1$ and $G_2$ have dimension smaller than G, we can proceed by induction
\begin{align*}
d(\Gamma) & \leq d(\Gamma_1) + d(\Gamma_2) \leq C_1 \text{vol}(G_1 / \Gamma_1) + C_2 \text{vol}(G_2 / \Gamma_2)  \\ &  \leq C_3 \text{vol}(G_1/\Gamma_1) \text{vol}(G_2/\Gamma_2)  \le C \text{vol}(G/\Gamma)
\end{align*}
Here the third inequality is due to the fact that these quantities are bounded from below by $2$ as this is the minimal number of generators of a lattice in a non-compact semi-simple Lie group. The fourth inequality is due to the fact that $G$ admits finitely many product decompositions. 
\end{proof}

\subsection{Combining the cases}

The proof of Theorem \ref{main_theorem} is now straightforward by combining the ingredients we have proved so far, that is, Theorem \ref{rank_in_amenable_radical_is_bounded}  and Theorem \ref{reducible_generalization}.

\begin{proof}[Proof of Theorem \ref{main_theorem}]
Let $\Gamma \leq G$ be a lattice, then one obtains the exact sequence 
\[ 
1 \rightarrow \Gamma \cap A \rightarrow \Gamma \rightarrow \bigslant{\Gamma}{\Gamma \cap A} \rightarrow 1
\]
so 
$$
 d(\Gamma) \leq  d(\bigslant{\Gamma}{\Gamma \cap A}) + d(\Gamma \cap A)  \leq C \operatorname{covol}_s(\Gamma) + b.
$$

\end{proof}

\begin{rmk}
Recall that by the Kazhdan--Margulis theorem $\operatorname{covol}_s(\Gamma)$ is bounded from below. Thus one can omit the additive constant $b$ from the statement of Theorem \ref{main_theorem} by enlarging the multiplicative one.
\end{rmk}

\end{document}